\documentclass[12pt,3p]{elsarticle}

\usepackage{amsmath}
\usepackage{amssymb}
\usepackage{amsthm}
\theoremstyle{plain}
\newtheorem{thm}{Theorem}[section]
\newtheorem{lem}[thm]{Lemma}
\newtheorem{cor}[thm]{Corollary}
\newtheorem{prop}[thm]{Proposition}
\theoremstyle{definition}
\newtheorem{defn}[thm]{Definition}
\newtheorem{problem}[thm]{Problem}
\theoremstyle{remark}
\newtheorem{rem}[thm]{Remark}
\newcommand{\Eq}[1]{\eqref{eq:#1}}

\makeatletter
\def\rom#1{\mbox{\leavevmode\skip@\lastskip\unskip\/%
           \ifdim\skip@=\z@\else\hskip\skip@\fi{\rm{#1}}}}
\makeatother

\makeatletter
\def\ps@pprintTitle{%
     \let\@oddhead\@empty
     \let\@evenhead\@empty
     \def\@oddfoot{%
       \hfill\thepage \hfill}%
     \let\@evenfoot\@oddfoot}
\makeatother
\newcommand{\D}{\mathbb{D}}

\newcommand{\N}{\mathbb{N}}

\newcommand{\R}{\mathbb{R}}

\newcommand{\cD}{\mathcal{D}}
\newcommand{\cH}{\mathcal{H}}
\newcommand{\cL}{\mathcal{L}}
\newcommand{\cP}{\mathcal{P}}
\newcommand{\cS}{\mathcal{S}}
\newcommand{\fB}{\mathfrak{B}}
\newcommand{\fM}{\mathfrak{M}}

\newcommand{\gm}{\gamma}\newcommand{\dl}{\delta}
\newcommand{\eps}{\varepsilon}
\newcommand{\te}{\theta}

\newcommand{\sg}{\sigma}
\newcommand{\ph}{\varphi}

\newcommand{\Gm}{\Gamma}

\newcommand{\Om}{\Omega}
\newcommand{\bone}{\mathbf{1}}

\newcommand{\dv}{\mathop{\mathrm{div}}\nolimits}
\newcommand{\vol}{\mathop{\mathrm{vol}}\nolimits}
\newcommand{\diam}{\mathop{\mathrm{diam}}\nolimits}
\newcommand{\la}{\langle}
\newcommand{\ra}{\rangle}
\newcommand{\FCb}{\mathcal{F}C_b}
\newcommand{\nab}{\nabla}
\newcommand{\de}{\partial}
\numberwithin{equation}{section}
\begin{document}

\begin{frontmatter}



\title{Sets of finite perimeter and the Hausdorff--Gauss measure\\ on the Wiener space\tnoteref{t1}}
\tnotetext[t1]{This study was supported by a Grant-in-Aid for Young Scientists (B) (18740070, 21740094).}

\author{Masanori Hino}
\ead{hino@i.kyoto-u.ac.jp}
\address{Graduate School of Informatics,
Kyoto University,
Kyoto 606-8501, Japan}
\begin{abstract}
In Euclidean space,
the integration by parts formula for a set of finite perimeter is
expressed by the integration with respect to a type of surface measure.
According to geometric measure theory, this surface measure is realized by the
one-codimensional Hausdorff measure restricted on the reduced boundary and/or the measure-theoretic boundary, which may be strictly smaller than the topological boundary.
In this paper, we discuss the counterpart of this measure in the abstract Wiener space, which is a typical infinite-dimensional space.
We introduce the concept of the measure-theoretic boundary in the Wiener space and provide the integration by parts formula for sets of finite perimeter. 
The formula is presented in terms of the integration with respect to the
one-codimensional Hausdorff--Gauss measure restricted on the measure-theoretic boundary.
\end{abstract}


\end{frontmatter}


\section{Introduction}
The concept of functions of bounded variation on a domain of $\R^m$ is a fundamental concept in  geometric measure theory.
Let $U$ be a domain of $\R^m$. 
By definition, a real-valued Lebesgue integrable function $f$ on $U$ has {\em bounded variation} if
\[
\sup\left\{\left.\int_U (\dv G)f\,dx\,\right|\,
G\in C_c^1(U\to\R^m),\ |G(x)|_{\R^m}\le1
\mbox{ for all }x\in U\right\}
<\infty,
\]
where $C_c^1(U\to\R^m)$ denotes the set of all $\R^m$-valued functions $G$ on $U$ such that $G$ is continuously differentiable and $G$ vanishes outside a certain compact subset of $U$, and $|\cdot|_{\R^m}$ denotes the Euclidean norm on $\R^m$.
One of the basic properties of a function $f$ of bounded variation on $U$ is that there exist a positive Radon measure $\nu$ on $U$ and a measurable function $\sg\colon U\to\R^m$ such that $|\sg(x)|_{\R^m}=1$ $\nu$-a.e.\,$x$ and
\begin{equation}\label{eq:ibp}
  \int_U (\dv G)f\,dx=-\int_U \la G,\sg\ra_{\R^m}\,d\nu
  \quad\mbox{for all }G\in C_c^1(U\to\R^m),
\end{equation}
where $\la \cdot,\cdot\ra_{\R^m}$ denotes the standard inner product on $\R^m$.
This follows directly from the Riesz representation theorem.
Roughly speaking, we can say that $f$ has an $\R^m$-valued measure $\sg\,d\nu$ as the weak gradient.
A Lebesgue measurable subset $A$ of $U$ is called a {\em set of finite perimeter} or sometimes a {\em Caccioppoli set} in $U$ if the indicator function $\bone_A$ of $A$ has bounded variation on $U$.
Then, Eq.~\Eq{ibp} is rewritten as
\begin{equation}\label{eq:ibp2}
  \int_A \dv G\,dx=-\int_{\partial A} \la G,\sg\ra_{\R^m}\,d\nu
  \quad\mbox{for all }G\in C_c^1(U\to\R^m),
\end{equation}
since the support of $\nu$ is proved to be a subset of the topological boundary $\partial A$ of $A$.
When $A$ is a bounded domain with a smooth boundary, Eq.~\Eq{ibp2} is identical to the Gauss--Green formula, and $\sg$ and $\nu$ are expressed as the unit inner normal vector field on $\partial A$ and the surface measure on $\partial A$, respectively.
Although $\partial A$ is not smooth in general, the deep theorem known as the structure theorem in geometric measure theory guarantees that $A$ has a ``measure-theoretical $C^1$-boundary.''
To state this claim more precisely, let us define the {\em reduced boundary} $\partial^\star A$ of $A$, which is a subset of $\partial A$, by the set of all points $x\in\R^m$ such that
\begin{enumerate}
\item $\nu(B(x,r))>0$ for all $r>0$;
\item $\displaystyle\lim_{r\to0}\frac{1}{\nu(B(x,r))}\int_{B(x,r)}\sg\,d\nu=\sg(x)$;
\item $|\sg(x)|_{\R^m}=1$.
\end{enumerate}
Here, $B(x,r)=\{y\in\R^m\mid |y-x|_{\R^m}\le r\}$.
Further, the {\em measure-theoretic boundary} $\partial_\star A$ of $A$ is defined as
\begin{equation}\label{eq:mtb}
\partial_\star A=\left\{x\in\R^m\left|\,
\limsup_{r\to0}\frac{\cL^m(B(x,r)\cap A)}{r^m}>0
\mbox{ and }
\limsup_{r\to0}\frac{\cL^m(B(x,r)\setminus A)}{r^m}>0\right\}\right.,
\end{equation}
where $\cL^m$ is the $m$-dimensional Lebesgue measure.
Then, the following theorems hold.
\begin{thm}[Structure theorem]\label{th:structure}
\begin{enumerate}
\item The measure $\nu$ is identified by the $(m-1)$-dimen\-sional (in other words, one-codimensional) Hausdorff measure $\cH^{m-1}$ restricted on $\partial^\star A$.
\item $\partial^\star A$ is decomposed as $\partial^\star A=\bigcup_{i=1}^\infty C_i\cup N$, where $\nu(N)=0$ and each $C_i$ is a compact subset of some $C^1$-hypersurface $S_i$ $(i=1,2,\dots)$; moreover, $\sg|_{C_i}$ is normal to $S_i$ $(i=1,2,\dots)$.
\end{enumerate}
\end{thm}
\begin{thm}\label{th:measuretheoretic}
The following relations hold: $\partial^\star A\subset \partial_\star A\subset \partial A$ and $\cH^{m-1}( \partial_\star A\setminus\partial^\star A)=0$; in particular, $\nu$ is also equal to $\cH^{m-1}$ restricted on $\partial_\star A$.
\end{thm}
In this sense, the measure $\nu$ can be regarded as the surface measure on suitable boundaries of $A$.
See, e.g., \cite{EG,Gi} for the proof of these claims.
The proof is heavily dependent on the fact that the Lebesgue measure satisfies the volume-doubling property and that the closed balls in $\R^m$ are compact; the proof also requires effective use of covering arguments.

On the other hand, in \cite{Fu00a,FH01,Hi04,HU08}, a theory for functions of bounded variation on the abstract Wiener space, which is a typical infinite-dimensional space, has been developed in relation to stochastic analysis.
In this case, the whole space $E$ is a Banach space equipped with a Gaussian measure $\mu$ as an underlying measure, and the tangent space $H$ is a Hilbert space that is continuously and densely embedded in $E$, as in the framework of the Malliavin calculus.
Then, we can define the concepts of functions of bounded variation on $E$ and sets of finite perimeter in a similar manner, and thus, we can obtain integration by parts formulas that are analogous to \Eq{ibp} and \Eq{ibp2}.
The existence of the measure $\nu$ is proved by a version of the Riesz representation theorem in infinite dimensions. This type of Riesz theorem was proved in \cite{Fu99} by utilizing a probabilistic method together with the theory of Dirichlet forms and in \cite{Hi04} by using a purely analytic method. 
Since the construction of the measure $\nu$ is somewhat abstract, the geometric interpretation of $\nu$ associated with sets of finite perimeter has been unknown thus far.

In this article, we consider Borel sets $A$ of $E$ that have a finite perimeter and prove that the measure $\nu$ associated with $A$ as above, which is denoted by $\|A\|_E$ in this paper, is identified by the one-codimensional Hausdorff--Gauss measure restricted on the measure-theoretic boundary $\partial_\star A$ of $A$.
This Hausdorff--Gauss measure on the Wiener space has been introduced in \cite{FL92} (see also \cite{Fe01}) in order to discuss the coarea formula on the Wiener space and the smoothness of Wiener functionals. Further, for the first time,  the measure-theoretic boundary $\partial_\star A$ is introduced in this study as a natural generalization of that in Euclidean space.
This identification justifies the heuristic observation that $\|A\|_E$ can be considered as the surface measure of $A$.
Since Gaussian measures on $E$ do not satisfy the volume-doubling property and closed balls in $E$ are not compact when $E$ is infinite-dimensional, most techniques in geometric measure theory cannot be applied directly.
Instead, we adopt the finite-dimensional approximation and utilize some results from geometric measure theory in finite dimensions; this is a reasonable approach since both the Hausdorff--Gauss measure and the measure-theoretic boundary are defined as the limits of the corresponding objects of finite-dimensional sections.
The most crucial task in the proof of the main theorem (Theorem~\ref{th:main}) is to prove that the order of these two limits can be possibly interchanged in a certain sense. 
Since the limit in the definition of the measure-theoretic boundary is not monotone, this claim is not straightforward and the proof requires rather technical arguments.

The representation of $\|A\|_E$ by the Hausdorff--Gauss measure enables us to take advantage of the general properties of Hausdorff--Gauss measures
(\cite{FL92,Fe01}); we can deduce that $\|A\|_E$ does not charge any sets of zero $(r,p)$-capacity if $p>1$ and $rp>1$, where the $(r,p)$-capacity is defined in the context of the Malliavin calculus. In \cite{Hi04}, such a smoothness property was proved  by using a different method, and the similarity between this smoothness property and that of the one-codimensional Hausdorff--Gauss measure was pointed out.
Our results clarify this relationship further.

Surface measures in infinite dimensions have been studied in various frameworks and approaches, such as in \cite{Go72,Sk,Ku,He80,AM88,Bo90,FL92}D
For example, in the early study by \mbox{Goodman}~\cite{Go72}, surface measures and normal vector fields were provided explicitly for what are called $H$-$C^1$ surfaces in the Wiener space.
In the study by Airault and Malliavin~\cite{AM88}, the surface measures on the level sets of smooth and nondegenerate functions are realized by generalized Wiener functionals in the sense of Malliavin calculus.
In the paper by Feyel and de La Pradelle~\cite{FL92}, the Hausdorff--Gauss measures were introduced to represent the surface measures, which has a great influence on this article.
Although these apparently different expressions should be closely related one another, it does not seem evident to derive one formula from another one directly.
It would be an interesting problem to clarify such an involved situation.
In this study, in contrast to the preceding ones, the smoothness assumption is not explicitly imposed on the boundary of the set under consideration.
The author hopes that our study will be useful to develop geometric measure theory in infinite dimensions.

This paper is organized as follows. In section~2, we provide the framework as well as the necessary definitions and propositions and state the main theorem.
We provide the proof of this theorem in section~3.
In section~4, we present some additional results as concluding remarks.

\section{Framework and main results}
Henceforth, we denote the Borel $\sg$-field of $X$ by $\fB(X)$ for a topological space $X$.
Let $(E,H,\mu)$ be an abstract Wiener space. In other words, $E$ is a separable Banach space, $H$ is a separable Hilbert
space densely and continuously embedded in $E$, and $\mu$ is a
Gaussian measure on $(E,\fB(E))$ that satisfies
\[
  \int_E \exp\left(\sqrt{-1}\,l(z)\right)\mu(dz)=\exp\left(-|l|_H^2/2\right),
  \quad l\in E^*.
\]
Here, 
the topological dual space $E^*$ of $E$ is regarded as a subspace of $H$ by the natural inclusion $E^*\subset H^*$ and the identification 
$H^*\simeq H$. The inner product and the norm
of $H$ are denoted by $\la\cdot,\cdot\ra_H$ and $|\cdot|_H$, respectively.
We mainly deal with the case in which both $E$ and $H$ are infinite-dimensional.
However, if necessary, many concepts discussed below can be easily modified such that they are valid even in the finite-dimensional case.

Denote by $\fM(E)$ the completion of $\fB(E)$ by $\mu$.
We define the following function spaces:
\begin{align}
\FCb^1(E)&=\left\{u\colon E\to\R\left|\,\begin{array}{ll} u(z)=f(h_1(z),\ldots,h_n(z)),\ h_1,\ldots,h_n\in E^*,\\ 
f\in C_b^1(\R^n)\mbox{ for some }n\in\N \end{array}\!\!\right.\right\},\nonumber\\
\FCb^1(E\to X)&=\mbox{the linear span of }\{u(\cdot)l\mid u\in \FCb^1(E),\ l\in X\},
\label{eq:fcb2}
\end{align}
for a Banach space $X$.
Here, $C_b^1(\R^n)$ is the set of all bounded continuous functions on $\R^n$ that have continuous bounded derivatives.
For a separable Hilbert space $X$ and $f\in\FCb^1(E\to X)$, the $H$-derivative of $f$, denoted by $\nab f$, is a map from $E$ to $H\otimes X$ defined by the relation
\[
  \la\nab f(z),l\ra_H =(\de_l f)(z) \quad\mbox{for all } l\in E^*\subset H,
\]
where 
\[
(\de_l f)(z)=\lim_{\eps\to0}(f(z+\eps l)-f(z))/\eps, 
\quad l\in E^*\subset H\subset E.
\]
For each $G\in \FCb^1(E\to E^*)$, the (formal) adjoint $\nab^* G$ is defined by the following identity:
\[
  \int_E (\nab^* G) f\,d\mu = \int_E \la G,\nab f\ra_H \,d\mu
  \qquad\mbox{for all }f\in\FCb^1(E).
\]
\begin{defn}[\cite{FH01}]\label{def:BV}
We say that a real-valued $\fM(E)$-measurable function $f$ on $E$ is of {\em bounded variation} ($f\in BV(E)$) if $\int_E|f|((\log |f|)\vee0)^{1/2}\,d\mu<\infty$ and
\[
  V_E(f):=\sup_G \int_E (\nab^* G)f\,d\mu
\]
is finite, where $G$ is taken over all functions in $\FCb^1(E\to E^*)$ such that $|G(z)|_H\le 1$ for every $z\in E$.
A subset $A$ of $E$ is said to have a {\em finite perimeter} if the indicator function $\bone_A$ of $A$ belongs to $BV(E)$.
We denote $V_E(\bone_A)$ by $V_E(A)$.
\end{defn}

One of the basic theorems concerning the functions of bounded variation is the following:
\begin{thm}[{\cite[Theorem~3.9]{FH01}}]\label{th:FH}
  For each $f\in BV(E)$, there exists a finite Borel measure $\nu$ on $E$ and an $H$-valued Borel measurable function $\sg$ on $E$ such that $|\sg|_H=1$ $\nu$-a.e.\ and
\begin{equation}\label{eq:bv}
 \int_E(\nab^* G)f\,d\mu=\int_E \la G,\sg\ra_H \,d\nu
 \quad\mbox{for every }G\in \FCb^1(E\to E^*).
\end{equation}
 Also, $\nu$ and $\sg$ are uniquely determined in the following sense: if $\nu'$ and $\sg'$ are different from $\nu$ and $\sg$ and also satisfy relation \Eq{bv}, then $\nu=\nu'$ and $\sg(z)=\sg'(z)$ for $\nu$-a.e.\,$z$.
\end{thm}
There is no minus sign on the right-hand side of \Eq{bv}, in contrast to \Eq{ibp}; this minus sign is included in the definition of $\nab^*$. 
For an $A\in\fM(E)$ that has finite perimeter, the $\nu$ and $\sg$ associated with $f:=\bone_A$ in the theorem above are denoted by $\|A\|_{E}$ and $\sg_A$, respectively.
Then, it is proved that the support of $\|A\|_{E}$ is included in the topological boundary $\de A$ of $A$ in $E$.
In other words, \Eq{bv} is rewritten as follows: for every $G\in \FCb^1(E\to E^*)$,
\begin{equation}\label{eq:bv2}
 \int_A(\nab^* G)\,d\mu=\int_{\partial A} \la G,\sg_A\ra_H \,d\|A\|_E.
 \end{equation}
A more detailed assertion has been presented in \cite[Theorem~3.15]{FH01}.

In order to state the main theorem in this paper, we introduce the concept of the Hausdorff--Gauss measure on $E$, essentially following the procedure in \cite{FL92,Fe01}.
We begin with the finite-dimensional case.
Let $F$ be an $m$-dimensional subspace of $E^*\,(\subset H)$ with $m\ge1$.
By including the inner product induced from $H$ in the subspace $F$, we regard $F$ as an $m$-dimensional Euclidean space.
Let $A$ be a (not necessarily Lebesgue measurable) subset of $F$.
For $\eps>0$, we set
\begin{equation}\label{eq:SFeps}
  \cS_{F,\eps}^{m-1}(A)=\inf_{\{B_i\}_{i=1}^\infty }\sum_{i=1}^\infty \vol_{m-1}(B_i),
\end{equation}
where $\{B_i\}_{i=1}^\infty$ is taken over all countable coverings of $A$ such that each $B_i$ is an {\em open ball} of diameter less than $\eps$, and
\[
  \vol_{m-1}(B_i)=V_{m-1}\cdot\left(\frac{\diam(B_i)}2\right)^{m-1},\quad
  V_{m-1}=\frac{\pi^{(m-1)/2}}{\Gm((m-1)/2+1)}.
\]
Note that $V_{m-1}$ is equal to the volume of the unit ball in $\R^{m-1}$.
Then, define
\[
  \cS_F^{m-1}(A)=\lim_{\eps\downarrow0}\cS_{F,\eps}^{m-1}(A).
\]
$\cS_F^{m-1}$ is called the $(m-1)$-dimensional (or one-codimensional) {\em spherical Hausdorff measure} on $F$.
$\cS_F^{m-1}$ is an outer measure on $F$ and a measure on $(F,\fB(F))$.
We do not use the standard Hausdorff measure $\cH_F^{m-1}$ (namely, the measure obtained by removing the restriction that $B_i$ is an open ball in \Eq{SFeps}) for the reason explained in the remark that follows Proposition~\ref{prop:nondecreasing} below.
The spherical Hausdorff measure and the Hausdorff measure differ on some pathological Borel sets but coincide on good sets that are considered in this study.
For further details on this assertion, we refer to \cite[Section~2.10.6, Corollary~2.10.42, Theorem~3.2.26]{Fed}.

The one-codimensional Hausdorff--Gauss measure $\te_F^{m-1}$ on $F$ is defined as
\begin{equation}\label{eq:HG}
  \te_F^{m-1}(A)=\int_A^* (2\pi)^{-m/2}\exp\left(-\frac{|x|_{\R^m}^2}{2}\right)\cS_F^{m-1}(dx),\quad
  A\subset F.
\end{equation}
Here, $\int^*$ denotes the outer integral for the  case in which $A$ is not measurable.
Note that we adopt a terminology different from \cite{FL92,Fe01}.
$\te_F^{m-1}$ is also an outer measure on $F$ and a measure on $(F,\fB(F))$.

We now consider the infinite-dimensional case.
Let $F$ be a finite-dimensional subspace of $E^*$, and let $m=\dim F$.
Define a closed subspace $\tilde F$ of $E$ by 
\[
\tilde F=\{z\in E\mid x(z)=0 \text{ for every } x\in F\subset E^*\}.
\]
Then, $E$ is decomposed as a direct sum $F\dotplus\tilde F$, where $F$ is regarded as a subspace of $E$.
The canonical projection operators from $E$ onto $F$ and $\tilde F$ are denoted by $p_F$ and $q_F$, respectively.
In other words, they are given by
\[
  p_F(z)=\sum_{i=1}^m h_i(z)h_i,\quad
    q_F(z)=z-p_F(z),
\]
where $\{h_1,\ldots,h_m\}\subset E^*\subset H\subset E$ is an orthonormal basis of $F$ in $H$.
Let $\mu_F$ and $\mu_{\tilde F}$ be the image measures of $\mu$ by $p_F$ and $q_F$, respectively.
The measure space $(E,\mu)$ can be identified by the product measure space $(F,\mu_F)\times(\tilde F,\mu_{\tilde F})$.
We define $\fM(F)$ as the completion of $\fB(F)$ by $\mu_{F}$, and $\fM(\tilde F)$ as the completion of $\fB(\tilde F)$ by $\mu_{\tilde F}$.

For $A\subset E$ and $y\in\tilde F$, the {\em section} $A_y$ is defined as
\begin{equation}\label{eq:section}
A_y=\{x\in F\mid x+y\in A\}.
\end{equation}
Define
\[
\cD_F=\{A\subset E\mid \mbox{the map }\tilde F\ni y\mapsto \te_F^{m-1}(A_y)\in[0,\infty]\mbox{ is $\fM(\tilde F)$-measurable}\}
\]
and
\begin{equation}\label{eq:rhoF}
\rho_F(A)=\int_{\tilde F}\te_F^{m-1}(A_y)\,\mu_{\tilde F}(dy)\quad
\mbox{for }A\in\cD_F.
\end{equation}
Then, we have the following propositions.
\begin{prop}[cf.\ {\cite[Proposition~3]{FL92} or \cite[Corollary~2.3]{Fe01}}]\label{prop:suslin}
Every Suslin set of $E$ belongs to $\cD_F$, and $\rho_F$ is a measure on $(E,\fB(E))$.
\end{prop}

Fix a sequence $\{l_i\}_{i=1}^\infty\subset E^*\,(\subset H)$ such that $\{l_i\}_{i=1}^\infty$ is a complete orthonormal system of $H$.
For $m\in\N$, let $F_m$ be an $m$-dimensional subspace of $E^*\,(\subset H\subset E)$ defined as
\begin{equation}\label{eq:Fm}
F_m=\mbox{the linear span of $\{l_1,\dots, l_m\}$}.
\end{equation}
Set $\cD=\bigcap_{m=1}^\infty \cD_{F_m}$. Note that $\cD$ contains all Suslin sets of $E$; in particular, $\cD\supset \fB(E)$.
\begin{prop}[{\cite[Proposition~6]{FL92} or \cite[Proposition~3.2]{Fe01}}]
\label{prop:nondecreasing}
For any $A\in\cD$, $\rho_{F_m}(A)$ is nondecreasing in $m$.
\end{prop}
The essential part of the proof is contained in \cite[Section~2.10.27]{Fed}, where it is explained that such a monotonicity does not hold when we replace $\cS_F^{m-1}$ with the Hausdorff measure $\cH_F^{m-1}$ in \Eq{HG}.
From this proposition, we can define 
\[
\rho(A):=\lim_{m\to\infty}\rho_{F_m}(A)\quad
\mbox{for }A\in\cD.
\]
Then, $\rho$ is a (non-$\sg$-finite) measure on $(E,\fB(E))$.
Denote by $\fM^\rho(E)$ the completion of $\fB(E)$ by $\rho$.
\begin{prop}\label{prop:Mrho}
$\fM^\rho(E)\subset \cD$, and $\rho$ is a complete measure on $(E,\fM^\rho(E))$.
\end{prop}
This proposition is proved in the next section.
\begin{defn}[cf.\ {\cite[Definition~8]{FL92}, \cite[Definition~3.3]{Fe01}}]\label{def:Hausdorff}
We call $\rho$ the {\em one-codimensional Hausdorff--Gauss measure} on $E$.
\end{defn}
\begin{rem}
\begin{enumerate}
\item The measure $\rho$ may depend on the choice of $\{l_i\}_{i=1}^\infty$.
In the original studies~\cite{FL92,Fe01}, the supremum of $\rho(A)$ is taken over all possible choices of $\{l_i\}_{i=1}^\infty$ in order to define the one-codimensional Hausdorff--Gauss measure of $A$.
In this study, such a procedure is not carried out.
\item Similarly, for each $n\in\N$, we can define the $n$-codimensional Hausdorff--Gauss measure on $E$.
\end{enumerate}
\end{rem}
Next, we introduce the concept of the measure-theoretic boundary of a subset of $E$.
\begin{defn}\label{def:boundary}
Let $A$ be a subset of $E$ and let
$F$ be a finite-dimensional subspace of $E^*\,(\subset H)$. 
Denote $\dim F$ by $m$ and the $m$-dimensional Lebesgue outer measure on $F$ by $\cL^m$.
Define
\[
\partial_\star^F A:=\left\{ z\in E\,\left|
\begin{array}{l}\displaystyle
\limsup_{r\to0}\frac{\cL^m(B(p_F(z),r)\cap A_{q_F(z)})}{r^m}>0\mbox{ and }\\\displaystyle
\limsup_{r\to0}\frac{\cL^m(B(p_F(z),r)\setminus A_{q_F(z)})}{r^m}>0
\end{array}\!\!\right.\right\}.
\]
Here, $B(p_F(z),r)$ is a closed ball in $F$ with center $p_F(z)$ and radius $r$, and $A_{q_F(z)}$ is a section of $A$ at $q_F(z)$ that is defined as in \Eq{section}.
\end{defn}
For each $y\in\tilde F$, the relation $(\partial_\star^F A)_y=\partial_\star(A_y)$ holds, where the left-hand side is the section of $\partial_\star^F A$ at $y$ as in \Eq{section} and the right-hand side is the measure-theoretic boundary of $A_y$ as in \Eq{mtb}.
\begin{defn}
For $A\subset E$, the {\em measure-theoretic boundary} $\partial_\star A$ of $A$ is defined as
\[
\partial_\star A:=\liminf_{m\to\infty}\partial_\star^{F_m}A
=\bigcup_{n=1}^\infty\bigcap_{m=n}^\infty\partial_\star^{F_m}A.
\]
\end{defn}
It can be easily seen that $\partial_\star A$ is a subset of $\partial A$.
In general, the sequence $\{\partial_\star^{F_m}A\}_{m=1}^\infty$ is not monotone in $m$.
We also note that $\partial_\star A$ may depend on the choice of $\{l_i\}_{i=1}^\infty$; however, in the case of our study, the difference is negligible, as we infer from the comment that follows Theorem~\ref{th:main}.
\begin{prop}\label{prop:Borel}
If $A$ is a Borel subset of $E$, then $\partial_\star A$ is also a Borel set.
\end{prop}
The proof is left to the next section.
The following theorem is the main theorem of this article.
\begin{thm}\label{th:main}
Let $A$ be a Borel subset of $E$ that has a finite perimeter.
Then, $\|A\|_E$ coincides with $\rho$ restricted on $\partial_\star A$. 
More precisely, 
\[
  \|A\|_E(B)=\rho(B\cap \partial_\star A),\quad
  B\in \fB(E).
\]
In particular, Eq.~\Eq{bv2} can be rewritten as
\begin{equation}\label{eq:bv3}
 \int_A(\nab^* G)\,d\mu=\int_{\partial_\star A} \la G,\sg_A\ra_H \,d\rho.
 \end{equation}
Further, the measure $(\rho,\fM^\rho(E))$ restricted on $\partial_\star A$ coincides with the completion of the measure $(\|A\|_E,\fB(\partial_\star A))$ on $\partial_\star A$.
\end{thm}
As a consequence of this theorem, the symmetric difference of $\partial_\star A$ and $\partial_\star' A$ is a null set with respect to $\|A\|_E$, where $\partial_\star' A$ is the measure-theoretic boundary of $A$ with respect to another complete orthonormal system $\{l'_i\}_{i=1}^\infty$.
Indeed, by letting $B_1=\partial_\star' A\setminus \partial_\star A$ and $B_2=\partial_\star A\setminus \partial'_\star A$, and denoting the one-codimensional Hausdorff--Gauss measure with respect to $\{l'_i\}_{i=1}^\infty$ by $\rho'$, we have
\[
\|A\|_E(B_1)=\rho(B_1\cap \partial_\star A)=0
\quad\mbox{and}\quad
\|A\|_E(B_2)=\rho'(B_2\cap \partial'_\star A)=0.
\]
\section{Proof}
\begin{proof}[Proof of Proposition~\ref{prop:Mrho}]
Let $A\in \fM^\rho(E)$. Then, there exist $B,C\in\fB(E)$ such that $B\subset A\subset C$ and $\rho(C\setminus B)=0$.
Let $m\in\N$. From Proposition~\ref{prop:nondecreasing}, $\rho_{F_m}(C\setminus B)=0$, which, from Eq.~\Eq{rhoF} and the Fubini theorem, implies that $\te_{F_m}^{m-1}((C\setminus B)_y)=0$ for $\mu_{\tilde F_m}$-a.e.\,$y\in\tilde F_m$.
For such $y$, $\te_{F_m}^{m-1}(B_y)=\te_{F_m}^{m-1}(C_y)=\te_{F_m}^{m-1}(A_y)$.
Since 
\[
\tilde F_m\ni y\mapsto \te_{F_m}^{m-1}(B_y)\in[0,\infty]
\]
is $\fM(\tilde F_m)$-measurable, $\te_{F_m}^{m-1}(A_y)$ is also $\fM(\tilde F_m)$-measurable in $y\in\tilde F_m$.
Therefore, we have $A\in\cD_{F_m}$ and $\rho_{F_m}(A)=\rho_{F_m}(B)$.
Consequently, we conclude that $A\in\cD$ and $\rho(A)=\rho(B)$.
In particular, the measure space $(E,\fM^\rho(E),\rho)$ is the completion of $(E,\fB(E),\rho)$.
\end{proof}
\begin{proof}[Proof of Proposition~\ref{prop:Borel}]
It is sufficient to prove that $\partial_\star^F A$ in Definition~\ref{def:boundary} is a Borel set.
Let $r>0$.
Since the map 
\[
F\times F\times\tilde F\ni(x,w,y)\mapsto \bone_{B(x,r)}(w)\bone_A(w+y)\in\R
\]
is Borel measurable, from the Fubini theorem, the map 
\[
F\times\tilde F\ni(x,y)\mapsto\int_F\bone_{B(x,r)}(w)\bone_A(w+y)\,\cL^m(dw)\in\R
\]
is Borel measurable.
Therefore, $\cL^m(B(p_F(z),r)\cap A_{q_F(z)})$ is a Borel measurable function in $z\in E$.

We will prove that
\begin{align}
\lefteqn{\left\{z\in E\,\left|\,\limsup_{r\to0}\frac{\cL^m(B(p_F(z),r)\cap A_{q_F(z)})}{r^m}>0\right\}\right.}&\nonumber\\
&=\left\{z\in E\,\left|\,\limsup_{j\in\N,\,j\to\infty}\frac{\cL^m(B(p_F(z),2^{-j})\cap A_{q_F(z)})}{(2^{-j})^m}>0\right\}\right..
\label{eq:limsup}
\end{align}
Denote the left-hand side and the right-hand side by $B_1$ and $B_2$, respectively.
The inclusion $B_1\supset B_2$ is trivial. 
When $2^{-j-1}<r\le 2^{-j}$, 
\[
  \frac{\cL^m(B(x,r)\cap A_y)}{r^m}
  \le 2^m\cdot\frac{\cL^m(B(x,2^{-j})\cap A_y)}{(2^{-j})^m},\quad
  x\in F,\ y\in\tilde F.
\]
This implies that
\[
\limsup_{r\to0}\frac{\cL^m(B(p_F(z),r)\cap A_{q_F(z)})}{r^m}
\le 2^m\limsup_{j\in\N,\,j\to\infty}\frac{\cL^m(B(p_F(z),2^{-j})\cap A_{q_F(z)})}{(2^{-j})^m}.
\]
Therefore, $B_1\subset B_2$.
Hence, \Eq{limsup} holds.
The Borel measurability of this set results from the expression $B_2$.

Similarly, we can prove that the set $\{z\in E\mid\limsup_{r\to0} r^{-m}\cL^m(B(p_F(z),r)\setminus A_{q_F(z)})>0\}$ is also Borel measurable.
\end{proof}
The rest of this section is devoted to the proof of Theorem~\ref{th:main}.
We use the same notations as those used in the previous section.
In the following discussion, let $F$ be a finite-dimensional subspace of $E^*$ or $F=E$.
Let $K$ be a finite-dimensional subspace of $F\cap E^*$.
We regard $K$ as a subspace of $H$ and include the inner product induced from $H$ in $K$.
As a convention, $\mu$ is denoted by $\mu_F$ when $F=E$.
When $F$ is finite-dimensional,
we define $\FCb^1(F\to K)$ (resp.\ $\FCb^1(\tilde F\to K)$), as in \Eq{fcb2}, with respect to the abstract Wiener space $(F,F,\mu_F)$ (resp.\ $(\tilde F,\tilde F\cap H,\mu_{\tilde F})$).
In this case,
$\FCb^1(F\to K)$ is also denoted by $C_b^1(F\to K)$.
By abuse of notation, the gradient operator and its adjoint operator for both $(F,F,\mu_F)$ and $(\tilde F,\tilde F\cap H,\mu_{\tilde F})$ are denoted by the symbols $\nab$ and $\nab^*$, respectively, which are the same as those for $(E,H,\mu)$.

For $A\in \fM(F)$, we define
\[
 V_{F,K}(A)=\sup\left\{\left.\int_A\nab^*G\,d\mu_F\right|\,
 G\in\FCb^1(F\to K),\ |G(x)|_K\le 1\mbox{ for all }x\in F\right\}
 (\le\infty).
\]
\begin{prop}\label{prop:VFK}
Suppose $V_{F,K}(A)<\infty$.
Then, there exist a Borel measure $\|A\|_{F,K}$ on $F$ and a $K$-valued Borel measurable function $\sg_{A,F,K}$ on $F$ such that $\|A\|_{F,K}(F)=V_{F,K}(A)$, $|\sg_{A,F,K}(z)|_K=1$ for $\|A\|_{F,K}$-a.e.\,$z$, and
\begin{equation}\label{eq:caci}
  \int_A \nab^* G\,d\mu_F=\int_F\la G,\sg_{A,F,K}\ra_K\,d\|A\|_{F,K}
  \quad \mbox{for every }G\in\FCb^1(F\to K).
\end{equation}
Also, $\|A\|_{F,K}$ and $\sg_{A,F,K}$ are uniquely determined; in other words, if $\|A\|_{F,K}'$ and $\sg_{A,F,K}'$ are different from $\|A\|_{F,K}$ and $\sg_{A,F,K}$ and satisfy  relation \Eq{caci}, then $\|A\|_{F,K}=\|A\|_{F,K}'$ and $\sg_{A,F,K}(z)=\sg_{A,F,K}'(z)$ for $\|A\|_{F,K}$-a.e.\,$z$.
\end{prop}
\begin{proof}
The proof of this proposition is similar to that in \cite[Theorem~3.9]{FH01};
this proof is simpler since $K$ is finite-dimensional.

Let $k=\dim K$.
Select an orthonormal basis $h_1,\dots, h_k$ of $K$. 
Let $i=1,\dots, k$. Select $g\in \FCb^1(F)$ and let $G(\cdot)=g(\cdot)h_i\in \FCb^1(F\to K)$.
Then,
\[
  (\nab^* G)(z)=-(\partial_{h_i}g)(z)+g(z)h_i(z).
\]
From \cite[Theorem~2.1]{FH01} and the argument in the first part of the proof of \cite[Theorem~3.9]{FH01}, there exists a signed Borel measure $D_A^i$ on $F$ such that
\begin{equation}\label{eq:DAi1}
  \int_F (\nab^*G) \bone_A\,d\mu_F
  = \int_F g\,dD_A^i\qquad
  \mbox{for all }G(\cdot)=g(\cdot)h_i\in \FCb^1(F\to K).
\end{equation}
Define $\Delta_A:=\sum_{i=1}^k|D_A^i|$, where $|D_A^i|$ is the total variation measure of $D_A^i$.
For each $i$, denote the Radon--Nikodym derivative $dD_A^i/d\Delta_A$ by $\gm_i$. 
We may assume that $\gm_i$ is Borel measurable.
Define a Borel measure $\|A\|_{F,K}$ on $F$ and a $K$-valued Borel measurable function $\sg_{A,F,K}$ on $F$ as
\begin{align}
\|A\|_{F,K}(dz)&=\sqrt{\sum_{j=1}^k\gm_j(z)^2}\,\Delta_A(dz),\label{eq:AFK}\\
\sg_{A,F,K}(z)&=\begin{cases}\displaystyle
\sum_{i=1}^k\frac{\gm_i(z)}{\sqrt{\sum_{j=1}^k\gm_j(z)^2}}h_i
& \mbox{if }\displaystyle\sum_{j=1}^k\gm_j(z)^2\ne0\\
\hss 0
& \mbox{if }\displaystyle\sum_{j=1}^k\gm_j(z)^2=0
\end{cases}.
\label{eq:sAFK}
\end{align}
Then, for any $i=1,\dots, k$ and $g\in \FCb^1(F)$,
\[
\int_F g\,dD_A^i
=\int_F g\,\la h_i,\sg_{A,F,K}\ra_K\,d\|A\|_{F,K}.
\]
We obtain \Eq{caci} by
combining this equation with \Eq{DAi1}.
By construction, $|\sg_{A,F,K}(z)|_K=1$ for $\|A\|_{F,K}$-a.e.\,$z\in F$.
It is evident from \Eq{caci} that the inequality $V_{F,K}(A)\le\|A\|_{F,K}(F)$ holds.
To prove the converse inequality, it is sufficient to select a sequence $\{G_n\}_{n=1}^\infty$ from $\FCb^1(F\to K)$ such that $|G_n(z)|_K\le 1$ for all $z\in F$ and $\lim_{n\to \infty}G_n(z)=\sg_{A,F,K}(z)$ for $\|A\|_{F,K}$-a.e.\,$z$.

The uniqueness is proved in the same manner as in the proof of \cite[Theorem~3.9]{FH01}.
\end{proof}

Henceforth, let $F$ be a finite-dimensional subspace of $E^*$ and $K$ be a subspace of $F$.
As in previous sections, both $F$ and $K$ are regarded as subspaces of $H$ as well as $E^*$.\begin{prop}\label{prop:V}
Let $A\in\fM(E)$ with $V_{E,K}(A)<\infty$.
Then, the map $\tilde F\ni y\mapsto V_{F,K}(A_y)\in[0,\infty]$ is $\fM(\tilde F)$-measurable, and
\begin{equation}\label{eq:V}
  \int_{\tilde F} V_{F,K}(A_y)\,\mu_{\tilde F}(dy)\le V_{E,K}(A).
\end{equation}
In particular, $V_{F,K}(A_y)<\infty$ for $\mu_{\tilde F}$-a.e.\,$y\in \tilde F$.
Here, $A_y$ is a section of $A$ that is defined in \Eq{section}.
\end{prop}
\begin{rem}
In fact, equality holds in \Eq{V}. This will be proved in Proposition~\ref{prop:slice}~(iii).
\end{rem}
\begin{proof}[Proof of Proposition~$\ref{prop:V}$.]
Let $\D^{1,2}(F\to K)$ denote the $(1,2)$-Sobolev space of $K$-valued functions on $F$; in other words, it is the completion of $C_b^1(F\to K)$ with respect to the norm $\|f\|_{\D^{1,2}}:=\left(\int_F (|\nab f|_{F\otimes K}^2+|f|_K^2)\,d\mu_F\right)^{1/2}$.
This is a Hilbert space with the inner product $\la f,g\ra_{\D^{1,2}}:=\int_F (\la\nab f,\nab g\ra_{F\otimes K}+\la f,g\ra_K)\,d\mu_F$.
Select a sequence $\{f_j\}_{j=1}^\infty$ from $C_b^1(F\to K)$ such that the following hold:
\begin{itemize}
\item $|f_j(x)|_K\le1$ for all $j\in\N$ and $x\in F$.
\item The set $\{f_j\mid j\in\N\}$ is dense in $\{g\in\D^{1,2}(F\to K)\mid |g(x)|_K\le 1\ \mbox{for $\mu_F$-a.e.}\,x\}$ with the topology of $\D^{1,2}(F\to K)$.
\end{itemize}
For any $B\in\fM(F)$, we have
\[
  V_{F,K}(B)=\sup_{j\in\N}\int_B \nab^* f_j(x)\,\mu_F(dx),
\]
since $\nab^* $ extends to a continuous operator from $\D^{1,2}(F\to K)$ to $L^2(F)$.
For $f\in C_b^1(F\to K)$, the map
\[
 E\simeq F\times \tilde F\ni (x,y)\mapsto \bone_A(x+y)\nab^* f(x)\in\R
\]
is $\fM(E)$-measurable. By the Fubini theorem, $A_y\in\fM(F)$ for $\mu_{\tilde F}$-a.e.\,$y\in \tilde F$, and the map
\[
  \tilde F\ni y\mapsto \int_{A_y}\nab^* f(x)\,\mu_F(dx)\in \R
\]
is $\fM(\tilde F)$-measurable. 
Therefore, the map
$\tilde F\ni y\mapsto V_{F,K}(A_y)\in[0,\infty]$ is also $\fM(\tilde F)$-measurable.

Let $\eps>0$.
We inductively define a sequence $\{C_j\}_{j=0}^\infty$ of subsets of $\tilde F$ as follows:
\begin{align*}
C_0&=\emptyset,\\
C_j&=\left\{y\in\tilde F\left|\, A_y\in\fM(F)\mbox{ and }
\int_{A_y}\nab^* f_j(x)\,\mu_F(dx)\ge(1-\eps)V_{F,K}(A_y)\wedge\eps^{-1}\right.\right\}\mbox{\Huge$\setminus$}\,\bigcup_{i=0}^{j-1}C_i,\\
&\qquad j=1,2,\dots.
\end{align*}
Then, $C_j\in\fM(\tilde F)$ for all $j$ and $\mu_{\tilde F}(\tilde F\setminus \bigcup_{j=1}^\infty C_j)=0$.

Let $n\in\N$. Define $D_n=\bigcup_{j=1}^n C_j$ and
\[
g_n(x,y)=\sum_{j=1}^n f_j(x) \bone_{C_j}(y)
\quad\mbox{for }(x,y)\in F\times \tilde F\simeq E.
\]
We also regard $g_n$ as an element of $L^2(\tilde F\to\D^{1,2}(F\to K))$ by the map 
\[
\tilde F\ni y\mapsto(x\mapsto g_n(x,y))\in \D^{1,2}(F\to K).
\]
Since $\FCb^1(\tilde F\to\D^{1,2}(F\to K))$ is dense in $L^2(\tilde F\to\D^{1,2}(F\to K))$ and $C_b^1(F\to K)$ is dense in $\D^{1,2}(F\to K)$,  $\FCb^1(\tilde F\to C_b^1(F\to K))$ is dense in $L^2(\tilde F\to\D^{1,2}(F\to K))$.
Therefore, we can select a sequence $\{u_j\}_{j=1}^\infty$ from $\FCb^1(\tilde F\to C_b^1(F\to K))$ -- also considered a subspace of $\FCb^1(E\to K)$ -- such that 
\begin{itemize}
\item $u_j\to g_n$ in $L^2(\tilde F\to\D^{1,2}(F\to K))$ and $\mu_{\tilde F}$-a.e.\ as $j\to\infty$,
\item $|u_j(x,y)|_K\le 1$ for all $j\in\N$ and $(x,y)\in F\times \tilde F$.
\end{itemize}
For $\mu_{\tilde F}$-a.e.\,$y\in\tilde F$, we have
\begin{align*}
\lim_{j\to\infty}\int_{A_y}\nab^*(u_j(\cdot,y))\,d\mu_F
&=\int_{A_y}\nab^*(g_n(\cdot,y))\,d\mu_F\\
&=\begin{cases}\displaystyle
\int_{A_y}\nab^* f_m\,d\mu_F &\mbox{if $y\in C_m$ for some $m=1,\dots,n$}\\
\hss0&\mbox{if }y\notin D_n
\end{cases}\\
&\ge ((1-\eps)V_{F,K}(A_y)\wedge \eps^{-1})\cdot \bone_{D_n}(y).
\end{align*}
Therefore,
\begin{align}
\int_{D_n}\{(1-\eps)V_{F,K}(A_y)\wedge \eps^{-1}\}\,\mu_{\tilde F}(dy)
&\le \int_{\tilde F}\left(\lim_{j\to\infty}\int_{A_y}\nab^*(u_j(\cdot,y))\,d\mu_F\right)\mu_{\tilde F}(dy)
\nonumber\\
&= \lim_{j\to\infty}\int_{\tilde F}\left(\int_{A_y}\nab^*(u_j(\cdot,y))\,d\mu_F\right)\mu_{\tilde F}(dy)\nonumber\\
&=\lim_{j\to\infty}\int_A\nab^* u_j\,d\mu\nonumber\\
&\le V_{E,K}(A).\label{eq:VEKA}
\end{align}
Here, to obtain the equality in the second line, we used the uniform integrability of the sequence $\left\{\int_{A_y}\nab^*(u_j(\cdot,y))\,d\mu_F\right\}_{j=1}^\infty$, which follows from 
\[
\sup_{j\in\N}\int_{\tilde F}\left(\int_{A_y}\nab^*(u_j(\cdot,y))\,d\mu_F\right)^2\mu_{\tilde F}(dy)
\le \sup_{j\in\N}\int_{\tilde F}\left(\int_{F}(\nab^*(u_j(\cdot,y)))^2\,d\mu_F\right)\mu_{\tilde F}(dy)<\infty.
\]
To obtain the equality in the third line in \Eq{VEKA}, we used the identity $(\nab^*(u_j(\cdot,y)))(x)=(\nab^*u_j)(x,y)$, which follows from the assumption that $K$ is a subspace of $F$.
By letting $\eps\downarrow0$ and $n\to\infty$ in \Eq{VEKA}, we obtain \Eq{V}. 
\end{proof}
\begin{prop}\label{prop:slice}
Let $A\in\fM(E)$ with $V_{E,K}(A)<\infty$.\begin{enumerate}
\item Let $f$ be a bounded Borel measurable function on $E$.
Then, the map
\[
  \tilde F\ni y\mapsto \int_F f(x+y)\,\|A_y\|_{F,K}(dx)\in\R
\]
is $\fM(\tilde F)$-measurable, and
\[
\int_E f\,d\|A\|_{E,K}
=\int_{\tilde F}\left(
\int_F f(x+y)\,\|A_y\|_{F,K}(dx)\right)\mu_{\tilde F}(dy).
\]
In particular, for any $B\in \fB(E)$, the map
\[
  \tilde F\ni y\mapsto \|A_y\|_{F,K}(B_y)\in[0,\infty)
\]
is $\fM(\tilde F)$-measurable, and
\begin{equation}\label{eq:AEKB}
\|A\|_{E,K}(B)
=\int_{\tilde F} \|A_y\|_{F,K}(B_y)\,\mu_{\tilde F}(dy).
\end{equation}
\item For $\mu_{\tilde F}$-a.e.\,$y\in\tilde F$,
\[
\sg_{A,E,K}(x+y)=\sg_{A_y,F,K}(x)
\quad\mbox{for }\|A_y\|_{F,K}\mbox{-a.e.}\,x\in F.
\]
\item In Eq.~\Eq{V}, equality holds.
\end{enumerate}
\end{prop}
\begin{proof}
We may assume that $A$ is a Borel set.
Let $k=\dim K$.
Select an orthonormal basis $h_1,\dots, h_k$ of $K$, as in the proof of Proposition~\ref{prop:VFK}.

Let $i=1,\dots, k$.
Define $K_i$ as a one-dimensional vector space spanned by $h_i$.
We denote $V_{E,K_i}$ and $V_{F,K_i}$ by $V_{E,i}$ and $V_{F,i}$, respectively. 
Define $D_A^i$ and $D_C^i$ for $C\in\fM(F)$ with $V_{F,i}(C)<\infty$ so that the relations of the type of Eq.~\Eq{DAi1} in the proof of Proposition~\ref{prop:VFK} hold.
Note that $|D_A^i|(E)=V_{E,i}(A)$ and $|D_C^i|(F)=V_{F,i}(C)$.

From Proposition~\ref{prop:V},
\begin{equation}\label{eq:Vi}
  \int_{\tilde F}V_{F,i}(A_y)\,\mu_{\tilde F}(dy)
  \le V_{E,i}(A) \le V_{E,K}(A).
\end{equation}
Let $g\in\FCb^1(E)$ and define $G(z)=g(z)h_i$.
Then,
\begin{align}\label{eq:fubini}
\int_E g\,dD_A^i
&=\int_A \nab^* G\,d\mu\nonumber\\
&=\int_{\tilde F}\left(\int_F \bone_{A_y}(x)(\nab^*(G(\cdot+y)))(x)\,\mu_F(dx)\right)\mu_{\tilde F}(dy)\nonumber\\
&=\int_{\tilde F}\left(\int_F g(x+y)\,D_{A_y}^i(dx)\right)\mu_{\tilde F}(dy).
\end{align}
In particular, the map
\begin{equation}\label{eq:measurable}
\tilde F\ni y\mapsto 
\int_F g(x+y)\,D_{A_y}^i(dx)\in\R
\end{equation}
is $\fM(\tilde F)$-measurable.
From the domination $\int_F |g(x+y)|\,|D_{A_y}^i|(dx)\le \sup_{z\in E}|g(z)|\cdot V_{F,i}(A_y)$, Eq.~\Eq{Vi}, and the monotone class theorem, Eq.~\Eq{fubini} and the $\fM(\tilde F)$-measurability of \Eq{measurable} hold for all bounded Borel measurable functions $g$ on $E$.
In particular, for $B\in \fB(E)$, by setting $g=\bone_B$, we have
\begin{align}\label{eq:DAiB}
D_A^i(B)
&=\int_{\tilde F}\left(\int_F \bone_B(x+y)\,D_{A_y}^i(dx)\right)\mu_{\tilde F}(dy)\nonumber\\
&=\int_{\tilde F}D_{A_y}^i(B_y)\,\mu_{\tilde F}(dy),
\end{align}
and the map $\tilde F\ni y\mapsto D_{A_y}^i(B_y)\in\R$ is $\fM(\tilde F)$-measurable.

Select a Borel set $S^i$ of $E$ such that $|D_A^i|(\cdot)=D_A^i(\cdot\cap S^i)-D_A^i(\cdot\setminus S^i)$ (the Hahn decomposition).
Then,
\begin{align*}
V_{E,i}(A)
&= |D_A^i|(E)=D_A^i(S^i)-D_A^i(E\setminus S^i)\\
&=\int_{\tilde F}\left(D_{A_y}^i(S_y^i)-D_{A_y}^i(F\setminus S_y^i)\right)\,\mu_{\tilde F}(dy)\quad\mbox{(from~\Eq{DAiB})}\\
&\le \int_{\tilde F}|D_{A_y}^i|(F)\,\mu_{\tilde F}(dy)\\
&=\int_{\tilde F}V_{F,i}(A_y)\,\mu_{\tilde F}(dy)\\
&\le V_{E,i}(A)
\qquad\mbox{(from \Eq{Vi})},
\end{align*}
where $S^i_y$ is the section of $S^i$, that is, $S^i_y=\{x\in F\mid x+y\in S^i\}$.
Therefore, the inequalities in the above equations can be replaced by equalities.
In particular, there exists a $\mu_{\tilde F}$-null set $\tilde N_i$ in $\fB(\tilde F)$ such that for all $y\in\tilde F\setminus \tilde N_i$, $V_{F,i}(A_y)<\infty$ and 
\[
|D_{A_y}^i|(F)=D_{A_y}^i(S_y^i)-D_{A_y}^i(F\setminus S_y^i),
\]
which implies that $|D_{A_y}^i|(\cdot)=D_{A_y}^i(\cdot\cap S_y^i)-D_{A_y}^i(\cdot\setminus S_y^i)$; this provides the Hahn decomposition of $D_{A_y}^i$.
Then, for any $B\in\fB(E)$,
\begin{align*}
|D_A^i|(B)
&= D_A^i(B\cap S^i)-D_A^i(B\setminus S^i)\\
&=\int_{\tilde F}\left(D_{A_y}^i(B_y\cap S_y^i)-D_{A_y}^i(B_y\setminus S_y^i)\right)\mu_{\tilde F}(dy)\\
&=\int_{\tilde F}|D_{A_y}^i|(B_y)\,\mu_{\tilde F}(dy).
\end{align*}
Let $\tilde N=\bigcup_{i=1}^k \tilde N_i$.
Define $\Delta_A=\sum_{i=1}^k |D_A^i|$ and $\Delta_{A_y}=\sum_{i=1}^k |D_{A_y}^i|$, which can be defined for $y\in \tilde F\setminus \tilde N$.
Then, $\bone_{\tilde F\setminus \tilde N}(y)\cdot \Delta_{A_y}(B_y)$ is Borel measurable in $y\in\tilde F$ and
\[
  \Delta_A(B)=\int_{\tilde F}\Delta_{A_y}(B_y)\,\mu_{\tilde F}(dy).
\]
This implies that for any bounded Borel function $f$ on $E$,
\begin{equation}\label{eq:measiny}
\bone_{\tilde F\setminus \tilde N}(y)\cdot\int_F f(x+y)\, \Delta_{A_y}(dx)
\mbox{ is Borel measurable in }y\in\tilde F
\end{equation}
and
\begin{equation}\label{eq:DeltaA}
\int_E f\,d\Delta_A
=\int_{\tilde F}\left(\int_F f(x+y)\,\Delta_{A_y}(dx)\right)\mu_{\tilde F}(dy).
\end{equation}

For $z\in E$, let $x=p_F(z)\in F$ and $y=q_F(z)\in\tilde F$.
Let $B(x,r)=\{w\in F\mid |w-x|_F\le r\}$ for $r>0$, and define a function $\ph^i$ on $E$ for $i=1,\dots, k$ by
\[
\ph^i(z)=\begin{cases}\displaystyle
\limsup_{n\to\infty}\frac{D_{A_y}^i(B(x,1/n))}{D_{A_y}^i(B(x,1/n)\cap S_y^i)-D_{A_y}^i(B(x,1/n)\setminus S_y^i)}&\mbox{if }y\in \tilde F\setminus \tilde N\\
\hss 0& \mbox{if }y\in \tilde N\end{cases},
\]
where $0/0=+\infty$ by definition.
Then, from the differentiation theorem (see, e.g., \cite[Section~1.6]{EG}), for $y\in\tilde F\setminus \tilde N$ (in particular, for $\mu_{\tilde F}$-a.e.\,$y$),
$\ph^i(x+y)$ is equal to the Radon--\allowbreak Nikodym derivative $(dD_{A_y}^i/d\Delta_{A_y})(x)$ for $\Delta_{A_y}$-a.e.\,$x\in  F$.

We will prove that $\ph^i(z)$ is Borel measurable in $z\in E$.
Let $g$ be a real-valued, bounded Borel measurable function on $F\times F\times \tilde F\simeq F\times E$ such that $g(x,\cdot)\in\FCb^1(E)$ for every $x\in F$.
Define 
\[
G_x(w,y)=g(x,w,y)h_i\quad\mbox{for }(x,w,y)\in F\times F\times \tilde F\simeq F\times E.
\]
Then, for $x\in F$ and $y\in\tilde F\setminus \tilde N$,
\begin{align*}
\int_F g(x,w,y)\,dD_{A_y}^i(dw)
&=\int_{A_y}(\nab^* (G_x(\cdot,y)))(w)\,\mu_F(dw)\\
&=\int_F\bone_{A}(w+y)(\nab^* G_x)(w,y)\,\mu_F(dw).
\end{align*}
From the Fubini theorem, the map 
\[
F\times \tilde F\ni(x,y)\mapsto \bone_{\tilde F\setminus \tilde N}(y)\int_F g(x,w,y)\,dD_{A_y}^i(dw)\in\R
\]
is Borel measurable.
By the monotone class theorem,
this measurability holds for any bounded Borel measurable function $g$. 
By letting 
\begin{align*}
g(x,w,y)&=\bone_{B(x,1/n)}(w),\\
g(x,w,y)&=\bone_{B(x,1/n)}(w) \bone_{S^i}(w+y),
\mbox{ and}\\
g(x,w,y)&=\bone_{B(x,1/n)}(w) \bone_{E\setminus S^i}(w+y),
\end{align*}
we show that 
\begin{align*}
&\bone_{\tilde F\setminus \tilde N}(y)D_{A_y}^i(B(x,1/n)),\\
&\bone_{\tilde F\setminus \tilde N}(y)D_{A_y}^i(B(x,1/n)\cap S_y^i), \mbox{ and}\\
&\bone_{\tilde F\setminus \tilde N}(y)D_{A_y}^i(B(x,1/n)\setminus S_y^i)
\end{align*}
are all Borel measurable in $(x,y)\in F\times\tilde F$.
Therefore, $\ph^i(z)$ is Borel measurable in $z\in E$.

Now, for any $B\in\fB(E)$,
\begin{align*}
D_A^i(B)&=\int_{\tilde F}D_{A_y}^i(B_y)\,\mu_{\tilde F}(dy)
\qquad \mbox{(from \Eq{DAiB})}\\
&=\int_{\tilde F}\left(\int_{B_y} \ph^i(x+y)\,\Delta_{A_y}(dx)\right)\mu_{\tilde F}(dy)\\
&=\int_B \ph^i\,d\Delta_A.
\qquad \mbox{(from \Eq{DeltaA})}
\end{align*}
Therefore, $\ph^i$ is equal to the Radon--Nikodym derivative $dD_A^i/d\Delta_A$.

From the construction of $\|A\|_{E,K}$ and $\|A_y\|_{F,K}$ by \Eq{AFK}, we have
\begin{align*}
  \|A\|_{E,K}(dz)&=\sqrt{\sum_{j=1}^k \ph^j(z)^2}\,\Delta_A(dz)
\intertext{and}
  \|A_y\|_{F,K}(dx)&=\sqrt{\sum_{j=1}^k \ph^j(x+y)^2}\,\Delta_{A_y}(dx),\quad y\in\tilde F\setminus \tilde N.
\end{align*}
By combining this with \Eq{measiny} and \Eq{DeltaA}, we prove that claim~(i) holds.
From expression \Eq{sAFK}, we have
\begin{align*}
\sg_{A,E,K}(z)&=\begin{cases}\displaystyle
\sum_{i=1}^k\frac{\ph^i(z)}{\sqrt{\sum_{j=1}^k\ph^j(z)^2}}h_i
& \mbox{if }\displaystyle\sum_{j=1}^k\ph^j(z)^2\ne0\\
\hss 0
& \mbox{if }\displaystyle\sum_{j=1}^k\ph^j(z)^2=0
\end{cases},\\
\sg_{A_y,F,K}(x)&=\begin{cases}\displaystyle
\sum_{i=1}^k\frac{\ph^i(x+y)}{\sqrt{\sum_{j=1}^k\ph^j(x+y)^2}}h_i
& \mbox{if }\displaystyle\sum_{j=1}^k\ph^j(x+y)^2\ne0\\
\hss 0
& \mbox{if }\displaystyle\sum_{j=1}^k\ph^j(x+y)^2=0
\end{cases}
\quad\mbox{for }y\in\tilde F\setminus \tilde N.
\end{align*}
Therefore, claim~(ii) follows.
We obtain (iii) by letting $B=E$ in \Eq{AEKB}.
\end{proof}
\begin{prop}\label{prop:finite}
Let $A\in\fM(E)$ with $V_E(A)<\infty$.
Denote the orthogonal projection operator from $H$ to $F$ by $\pi_F$.
Then, $\sg_{A,E,F}(z)\,\|A\|_{E,F}(dz)=\pi_F\sg_A(z)\,\|A\|_E(dz)$.
In particular,
\begin{align*}
 \|A\|_{E,F}(dz)&=|\pi_F\sg(z)|_F\,\|A\|_E(dz),\\
  \sg_{A,E,F}(z)&=\begin{cases}\displaystyle
  \frac{\pi_F\sg_A(z)}{|\pi_F\sg_A(z)|_F}&\mbox{if } \pi_F\sg_A(z)\ne0\\
  \hss 0&\mbox{if } \pi_F\sg_A(z)=0
  \end{cases}
  \quad \mbox{for $\|A\|_{E,F}$-a.e.\,$z\in E$},
\end{align*}
and for every $B\in\fB(E)$, $\|A\|_{E,F_m}(B)$ increases to $\|A\|_E(B)$ as $m\to\infty$, where $\{F_m\}_{m=1}^\infty$ is defined as in \Eq{Fm}.
\end{prop}
\begin{proof}
Let $G\in\FCb^1(E\to F)$.
Then, from Proposition~\ref{prop:slice} with $K=F$,
\begin{align*}
\int_E \la G,\sg_{A,E,F}\ra_F\,d\|A\|_{E,F}
&=\int_{\tilde F}\left(\int_F\la G(x+y),\sg_{A,E,F}(x+y)\ra_F\,\|A_y\|_{F,F}(dx)\right)\mu_{\tilde F}(dy)\\
&=\int_{\tilde F}\left(\int_F(\nab^* (G(\cdot+y)))(x)\bone_{A_y}(x)\,\mu_F(dx)\right)\mu_{\tilde F}(dy)\\
&=\int_E \nab^*G\cdot\bone_A\,d\mu\\
&=\int_E\la G,\sg_A\ra_H\,d\|A\|_E\\
&=\int_E\la G,\pi_F\sg_A\ra_F\,d\|A\|_E.
\end{align*}
This proves the assertion.
\end{proof}
Let $m=\dim F$.
For a subset $A$ of $F$, we define the measure-theoretic boundary $\partial_\star A$ of $A$ in $F$ by replacing $\R^m$ with $F$ in \Eq{mtb}.\begin{prop}\label{prop:finitedim}
Let $A\in\fM(F)$ satisfy $V_{F,F}(A)<\infty$.
Then, $\|A\|_{F,F}$ is equal to the one-codimensional Hausdorff--Gauss measure $\te_F^{m-1}$ restricted on $\partial_\star A$, and $\sg_{A,F,F}$ is equal to the $\sg$ obtained by replacing $U$ and $\R^m$ in \Eq{ibp2} with $F$.
\end{prop}
\begin{proof}
Define
\[
\xi(x)=(2\pi)^{-m/2}\exp(-|x|_F^2/2),
\quad x\in F.
\]
Then, for $G\in C_c^1(F\to F)$ and $f\in C_c^1(F)$,
\begin{align*}
\int_F (\dv G)f\,d\cL^m
&=-\int_F\la G,\nab f\ra_F\,d\cL^m\\
&=-\int_F\la G\cdot \xi^{-1},\nab f\ra_F\,d\mu_F\\
&=-\int_F(\nab^*( G\cdot \xi^{-1})) f\,d\mu_F\\
&=-\int_F(\nab^*( G\cdot \xi^{-1}))\xi f\,d\cL^m.
\end{align*}
Therefore, $\dv G=-(\nab^*( G\cdot \xi^{-1}))\xi$.
This implies that $A$ has a locally finite perimeter in $F$ (with respect to the Lebesgue measure) in the following sense: for any bounded domain $U$ in $F\simeq\R^m$,
\[
\sup\left\{\left.\int_A (\dv G)\,d\cL^m\,\right|\,
G\in C_c^1(U\to F),\ |G(x)|_F\le1
\mbox{ for all }x\in U\right\}
<\infty.
\]
For such a set, Theorems~\ref{th:structure} and \ref{th:measuretheoretic} hold.
(See, e.g., Section~5.7.3, Theorem~2 and Section~5.8, Lemma~1 in \cite{EG}.)
In particular, the measure-theoretic boundary $\partial_\star A$ in $F$ is equal to a countable union of compact subsets of $C^1$-surfaces in $F$, up to an $\cH_F^{m-1}$-null set.
Here, $\cH_F^{m-1}$ is the $(m-1)$-dimensional Hausdorff measure on $F$.
Thus, any subset $B$ of $\partial_\star A$ with $\cH_F^{m-1}(B)<\infty$ is $(\cH_F^{m-1},m-1)$-rectifiable in the sense of \cite[Section~3.2.14]{Fed}.
From \cite[Theorem~3.2.26]{Fed}, $\cH_F^{m-1}(B)=\cS_F^{m-1}(B)$.
In other words, $\cH_F^{m-1}$ and $\cS_F^{m-1}$ coincide as (outer) measures on $\partial_\star A$.
Then, for $G\in C_c^1(F\to F)$,
\begin{align*}
\int_{\partial_\star A}\la G,\sg\ra_F\,d\te_F^{m-1}
&=\int_{\partial_\star A}\la G\xi,\sg\ra_F\,d\cH_F^{m-1}\\
&=-\int_A\dv(G\xi)\,d\cL^m
\quad\mbox{(from \Eq{ibp2} and Theorem~\ref{th:measuretheoretic})}\\
&=\int_A \nab^* G\,d\mu_F.\qquad(\mbox{because }\dv(G\xi)=-(\nab^* G)\xi)
\end{align*}
Therefore, $\|A\|_{F,F}$ is equal to the measure $\te_F^{m-1}$ restricted on $\partial_\star A$, and $\sg_{A,F,F}=\sg$.
\end{proof}
\begin{proof}[Proof of Theorem~\ref{th:main}]
Let $k,m\in\N$ with $m>k$.
Denote the linear span of $\{l_{k+1},l_{k+2},\dots,l_m\}$ by $F_m\ominus F_k$.
For $y_m\in \tilde F_m$ and $x\in F_m\ominus F_k$, let 
\[
(A_{y_m})_x=\{w\in F_k\mid x+w\in A_{y_m}\}\,(=\{w\in F_k\mid x+w+y_m\in A\}).
\]
Since $V_{F_m,F_m}(A_{y_m})<\infty$ for $\mu_{\tilde F_m}$-a.e.\,$y_m\in \tilde F_m$, for such $y_m$ and any $C\in\fB(F_m)$, we have
\[
\int_{F_m\ominus F_k}\|(A_{y_m})_x\|_{F_k,F_k}(C_x)\,\mu_{F_m\ominus F_k}(dx)
=\|A_{y_m}\|_{F_m,F_k}(C)\le \|A_{y_m}\|_{F_m,F_m}(C),
\]
by applying Proposition~\ref{prop:slice} to the abstract Wiener space $(F_m,F_m,\mu_{F_m})$.
By taking $C=(E\setminus \partial_\star^{F_m}A)_{y_m}\,(=F_m\setminus\partial_\star(A_{y_m}))$, from Proposition~\ref{prop:finitedim}, we have
\[
  0=\|A_{y_m}\|_{F_m,F_m}(C)
  \ge \int_{F_m\ominus F_k}\|(A_{y_m})_x\|_{F_k,F_k}(C_x)\,\mu_{F_m\ominus F_k}(dx).
\]
Then, we have
\begin{align*}
0&=\int_{\tilde F_m}\left(\int_{F_m\ominus F_k}\|(A_{y_m})_x\|_{F_k,F_k}(((E\setminus \partial_\star^{F_m}A)_{y_m})_x)\,\mu_{F_m\ominus F_k}(dx)\right)\mu_{\tilde F_m}(dy_m)\\
&=\int_{\tilde F_k}\|A_{y_k}\|_{F_k,F_k}((E\setminus\partial_\star^{F_m}A)_{y_k})\,\mu_{\tilde F_k}(dy_k).
\end{align*}
Therefore, for $\mu_{\tilde F_k}$-a.e.\,$y_k\in \tilde F_k$, $\partial_\star(A_{y_k})\subset (\partial_\star^{F_m}A)_{y_k}$ up to a $\|A_{y_k}\|_{F_k,F_k}$-null set, where $\partial_\star(A_{y_k})$ is the measure-theoretic boundary of $A_{y_k}$ in $F_k$.
By taking $\liminf_{m\to\infty}$,
\[
\partial_\star(A_{y_k})\subset \liminf_{m\to\infty}(\partial_\star^{F_m}A)_{y_k}
=\left(\liminf_{m\to\infty}\partial_\star^{F_m}A\right)_{y_k}
=(\partial_\star A)_{y_k}
\]
up to a $\|A_{y_k}\|_{F_k,F_k}$-null set.

Let $B\in\fB(E)$.
For $\mu_{\tilde F_k}$-a.e.\,$y_k\in\tilde F_k$, we have
\begin{align*}
\|A_{y_k}\|_{F_k,F_k}(B_{y_k})
&=\te_{F_k}^{k-1}(\partial_\star(A_{y_k})\cap B_{y_k})\qquad\mbox{(from Proposition~\ref{prop:finitedim})}\\
&\le\te_{F_k}^{k-1}\left((\partial_\star A)_{y_k}\cap B_{y_k}\right)\\
&=\te_{F_k}^{k-1}\left(((\partial_\star A)\cap B)_{y_k}\right).
\end{align*}
Integrating both sides with respect to $\mu_{\tilde F_k}(dy_k)$ and applying Proposition~\ref{prop:slice} with $K=F_k$, we have
\begin{equation}\label{eq:star1}
\|A\|_{E,F_k}(B)\le\rho_{F_k}((\partial_\star A)\cap B).
\end{equation}

On the other hand, by applying Proposition~\ref{prop:nondecreasing} with $(E,H,\mu)=(F_m,F_m,\mu_{F_m})$, for $\mu_{\tilde F_m}$-a.e.\,$y_m\in\tilde F_m$, we have
\begin{align*}
\int_{F_m\ominus F_k}\te_{F_k}^{k-1}((\partial_\star(A_{y_m})\cap B_{y_m})_x)\,\mu_{F_m\ominus F_k}(dx)
&\le \te_{F_m}^{m-1}(\partial_\star(A_{y_m})\cap B_{y_m})\\
&=\|A_{y_m}\|_{F_m,F_m}(B_{y_m}).
\end{align*}
Here, $(\partial_\star(A_{y_m})\cap B_{y_m})_x=\{w\in F_k\mid x+w\in \partial_\star(A_{y_m})\cap B_{y_m}\}$.
Then,
\begin{align*}
\rho_{F_k}((\partial_\star^{F_m}A)\cap B)
&=\int_{\tilde F_k}\te_{F_k}^{k-1}(((\partial_\star^{F_m}A)\cap B)_{y_k})\,\mu_{\tilde F_k}(dy_k)\\
&=\int_{\tilde F_m}\left(\int_{F_m\ominus F_k}\te_{F_k}^{k-1}((\partial_\star(A_{y_m})\cap B_{y_m})_x)\,\mu_{F_m\ominus F_k}(dx)\right)\mu_{\tilde F_m}(dy_m)\\
&\le\int_{\tilde F_m}\|A_{y_m}\|_{F_m,F_m}(B_{y_m})\,\mu_{\tilde F_m}(dy_m)\\
&= \|A\|_{E,F_m}(B) \qquad\mbox{(from Proposition~\ref{prop:slice})}\\
&\le \|A\|_E(B).
\end{align*}
From the Fatou lemma, we obtain
\begin{equation}\label{eq:star2}
\rho_{F_k}((\partial_\star A)\cap B)
\le
\liminf_{m\to\infty}\rho_{F_k}((\partial_\star^{F_m}A)\cap B)
\le \|A\|_E(B).
\end{equation}

From \Eq{star1} and \Eq{star2} and by letting $k\to\infty$, we have 
\[
\|A\|_E(B)\le\rho((\partial_\star A)\cap B)
\le \|A\|_E(B)
\]
by Proposition~\ref{prop:finite}.
Therefore, $\|A\|_E(B)=\rho((\partial_\star A)\cap B)$ for all $B\in\fB(E)$.
The final claim in Theorem~\ref{th:main} follows from the standard argument.
\end{proof}
\section{Concluding remarks}
\subsection{Remarks on $\|A\|_E$ and $\sg_A$}
Let $A$ be a subset of $E$ that has a finite perimeter. 
To state a further property of $\|A\|_E$, we recall the notion of Sobolev spaces and capacities on $E$ in the sense of the Malliavin calculus.
Let $K$ be a separable Hilbert space.
Let $\cP(E)$ be the set of all real-valued functions $u$ on $E$ that is expressed as $u(z)=g(h_1(z),\ldots,h_n(z))$ for some $n\in\N$, $h_1,\ldots,h_n\in E^*$, and some polynomial $g$ on $\R^n$.
Denote by $\cP(E\to K)$ the linear span of $\{u(\cdot)k\mid u\in \cP(E),\ k\in K\}$.
For $r\ge0$ and $p>1$, the $(r,p)$-Sobolev space $\D^{r,p}(E\to K)$ on $E$ is defined as the completion of $\cP(E\to K)$ by the (semi-)norm $\|\cdot\|_{r,p}$ defined by $\|f\|_{r,p}=\left(\int_E |(I-L)^{r/2}f|^p\,d\mu\right)^{1/p}$, where $L=-\nab^*\nab$ is the Ornstein--Uhlenbeck operator.
For $f\in \D^{r,p}(E\to K)$, $\|f\|_{r,p}$ is defined by continuity.
We denote $\D^{r,p}(E\to \R)$ by $\D^{r,p}(E)$.
The $(r,p)$-capacity $C_{r,p}$ on $E$ is defined as
\[
C_{r,p}(U)=\inf\{\|f\|_{r,p}^p\mid f\in \D^{r,p}(E)\mbox{ and }f\ge 1\ \mu\mbox{-a.e.\ on }U\}
\]
when $U$ is an open set of $E$ and
\[
C_{r,p}(B)=\inf\{C_{r,p}(U)\mid \mbox{$U$ is open and }B\subset U\}
\]
for a general $B\subset E$.
Then, from \cite[Theorem~4.4]{Fe01}, the $n$-codimensional Gauss--Hausdorff measure does not charge any set of $C_{r,p}$-null set if $rp>n$.
Therefore, by combining this fact with Theorem~\ref{th:main}, we have the following claims.
\begin{prop}\label{prop:smoothness}
Suppose that $p>1$ and $rp>1$.
Then, the measure $\|A\|_E$ does not charge any $C_{r,p}$-null set.
\end{prop}
This proposition has been proved in \cite{Hi04} (Proposition~4.6 and Remark~4.7) by using a different method.
Such a smoothness property of $\|A\|_E$ is important for the study of the stochastic analysis on $A$; refer to \cite{FH01} for further details on this topic.

A $K$-valued function $G$ on $E$ is called $C_{r,p}$-{\em quasicontinuous} if for any $\eps>0$, there exists an open set $U\subset E$ such that $C_{r,p}(U)<\eps$ and $G|_{E\setminus U}$ is continuous.
If $G=\tilde G$ $\mu$-a.e.\ and $\tilde G$ is $C_{r,p}$-quasicontinuous, we say that $\tilde G$ is a $C_{r,p}$-quasicontinuous modification of $G$.
In the manner similar to the proof of \cite[Lemma~4.3]{FH01}, it is not difficult to prove that every $G\in \D^{r,p}(E\to K)$ has a $C_{r,p}$-quasicontinuous modification $\tilde G$, and if a sequence $\{G_n\}_{n=1}^\infty$ converges to $G$ in $\D^{r,p}(E\to K)$, then there exists some $\{n_k\}\uparrow\infty$ such that $\tilde G_{n_k}$ converges to $\tilde G$ pointwise outside some $C_{r,p}$-null set.  
Using these facts, we can prove the following corollary.
\begin{cor}
For any $p>1$, Eq.~\Eq{bv3} is valid for any $G\in \D^{1,p}(E\to H)\cap L^\infty(E\to H)$, where $G$ in the right-hand side of \Eq{bv3} should be replaced by the $C_{1,p}$-quasicontinuous modification $\tilde G$.
\end{cor}
\begin{proof}
From the Meyer equivalence, $\nab^*$ extends to a continuous map from $\D^{1,p}(E\to H)$ to $L^p(E)$, and $\nab$ extends to a continuous map from $\D^{1,p}(E\to H)$ to $L^p(E\to H\otimes H)$; further, $\left\{\int_E(|f|_H^p+|\nab f|_{H\otimes H}^p)\,d\mu\right\}^{1/p}$ provides a norm on $\D^{1,p}(E\to H)$ that is equivalent to $\|\cdot\|_{1,p}$.
From a standard procedure,
we can take a sequence $\{G_n\}_{n=1}^\infty$ from $\FCb^1(E\to E^*)$ and a $C_{1,p}$-null set $N$ of $E$ such that $G_n$ converges to $G$ in $\D^{1,p}(E\to H)$ and $G_n(z)$ converges to $\tilde G(z)$ for all $z\in E\setminus N$, and $\sup\{G_n(z)\mid n\in\N,\ z\in E\}\vee \sup\{\tilde G(z)\mid z\in E\setminus N\}<\infty$.
Applying \Eq{bv3} to $G_n$ and letting $n\to\infty$, we obtain the conclusion.
\end{proof}

From Proposition~\ref{prop:finite}, the $H$-valued measure $\sg_A(z)\|A\|_E(dz)$ can be regarded as a kind of projective limit of the $H$-valued measures associated with finite-dimensional sections of $A$.
From the above fact and the structure theorem (Theorem~\ref{th:structure}), we can say that $\sg_A$ is described as the limit of normal vector fields on finite-dimensional sections of $A$. 
The determination of the validity of the infinite-dimensional version of the structure theorem is an open problem, which is stated below.
\begin{problem}
Does $\partial_\star A$ itself have an infinite-dimensional differential structure in a suitable sense, and can $\sg_A$ be interpreted as a normal vector field on $\partial_\star A$?
\end{problem}
If $A$ is given by the set $\{f>0\}$ for a nondegenerate function $f$ on $E$ that belongs to some suitable Sobolev space, then the answer is affirmative; see \cite{AM88,FL92,Fe01}.
In general, it does not seem that we can expect this type of a good expression for $A$.
Here, we present the typical examples under consideration.
Let $d\in\N$ and $(E,H,\mu)$ be the classical Wiener space on $\R^d$; in particular,
\begin{align*}
E&=\{w\in C([0,1]\to\R^d)\mid w(0)=0\},\\
H&=\left\{h\in E\left|\,
\mbox{$h$ is absolutely continuous and } \int_0^1 |\dot h(s)|_{\R^d}^2\,ds<\infty
\right.\!\right\},
\end{align*}
and $\mu$ is the law of the Brownian motion on $\R^d$ starting from $0$.
Let $\Om$ be a domain of $\R^d$ that includes $0$, and define 
\[
A=\{w\in E\mid w(t)\in\Om \mbox{ for all }t\in[0,1]\}.
\]
We say that $\Om$ satisfies the {\em uniform exterior ball condition} if there exists $\dl>0$ such that for every $y$ in the topological boundary of $\Om$ in $\R^d$, there exists $x\in\R^d\setminus \Om$ satisfying ${B}(x,\dl)\cap \overline{\Om}=\{y\}$, where ${B}(x,\dl)$ is the closed ball with center $x$ and radius $\dl$ and $\overline\Om$ is the closure of $\Om$.  
For example, bounded domains with boundaries in the $C^2$-class and convex domains satisfy this condition. 
Then, we have the following theorem.
\begin{thm}[{\cite[Theorem~5.1]{HU08}}]
Suppose $\Om$ satisfies the uniform exterior ball condition.
Then, $A$ is of finite perimeter.
\end{thm}
Further detailed properties are discussed in \cite{HU08} in a more general setting.
Sets of finite perimeter in the Wiener space appear in a natural manner as presented in \cite{HU08}, and in general, it seems difficult to treat such sets as level sets of smooth and nondegenerate functions.
 
\subsection{Remarks on measure-theoretic boundaries}
In general, $\partial_\star A$ is strictly smaller than $\partial A$.
A trivial example is a one point set.
It is natural to expect that $\partial_\star A$ coincides with $\partial A$ when $\partial A$ is smooth in a certain sense. We will state it as a problem as follows:
\begin{problem}
Provide sufficient conditions on $A$ such that $\partial_\star A=\partial A$. In particular, when $A$ is realized as $\{f>0\}$ for some function $f$ on $E$, what kind of condition on $f$ is sufficient to assure $\partial_\star A=\partial A$?
\end{problem}
As a partial answer, we will provide a simple sufficient condition at which $\partial_\star A=\partial A$ holds.
In the following discussion, $\{F_m\}_{m=1}^\infty$ is selected as in \Eq{Fm}.
For $A\subset E$, let $A^\circ$ and $\overline{A}$ denote the interior and the closure of $A$ in $E$, respectively.
\begin{prop}\label{prop:convex}
Suppose $A$ is a convex set of $E$ with $A^\circ\ne\emptyset$. 
Then, $\partial_\star A=\partial A$.
\end{prop}
For the proof, we state a basic result from convex analysis.
Let $G$ be a finite-dimensional affine space of $E$. 
For $C\subset G$, let $C^{\circ G}$, $\overline C^G$, and $\partial^G C$ be the interior, the closure, and the boundary of $C$ with respect to the relative topology of $G$, respectively.
\begin{lem}\label{lem:convex}
   Let $A$ be a convex set of $E$.  
If $A^\circ \cap G\ne \emptyset$, then $A^\circ\cap G=(A\cap G)^{\circ G}$, $\overline A\cap G=\overline{A\cap G}^G$, and $(\partial A)\cap G=\partial^G(A \cap G)$.
\end{lem}
\begin{proof}
Consider $y\in A^\circ \cap G$. We can choose an open ball $U$ with center $y$ that is included in $A^\circ$.

First, we prove $A^\circ\cap G\supset (A\cap G)^\circ$.
Consider $x\in (A\cap G)^\circ$.
There exists $s>0$ such that $w:=(1+s)x-sy\in (A\cap G)^\circ$.
Since $\frac1{1+s}w+\frac{s}{1+s}U$ is an open ball that includes $x$ and is included in $A$, we conclude that $x\in A^\circ$.
Since $x$ clearly belongs to $G$, we conclude that $x\in A^\circ\cap G$.

Next, we prove $\overline A\cap G\subset \overline{A\cap G}^G$.
Consider $x\in \overline A\cap G$. Then,
\[
\bigcup_{t\in(0,1]}\left((1-t)x+t(U\cap G)\right)\subset (A\cap G)^{\circ G},
\]
and $x$ is an accumulation point on the left-hand side; therefore, we have $x\in \overline{A\cap G}^G$.

Both the converse inclusions are obvious.
The last equality in the claim follows from the first two equalities.
\end{proof}
\begin{proof}[Proof of Proposition~\ref{prop:convex}]
It is sufficient to prove $\partial_\star A\supset\partial A$.
Let $F_\infty=\bigcup_{m=1}^\infty F_m$, which is a dense subspace of $E$.
Consider $z\in\partial A$.
By the assumption $A^\circ\ne\emptyset$, we have $A^\circ \cap(z+F_\infty)\ne\emptyset$.
Therefore, for sufficiently large $m$, $A^\circ \cap(z+F_m)\ne\emptyset$.
Denote $z+F_m$ by $G$.
From Lemma~\ref{lem:convex}, $z\in (\partial A)\cap G=\partial^G(A\cap G)$.
Since $A\cap G$ is a convex set, it has a Lipschitz boundary in $G$. (For the proof, see, e.g., \cite[Corollary~1.2.2.3]{Gr}.)
This implies $\partial^G(A\cap G)=(\partial_\star^{F_m}A)\cap G$, therefore $z\in \partial_\star^{F_m}A$.
Thus, $z$ belongs to $\partial_\star A$.
\end{proof}
{\em Note added in proof\/}: Two recent papers~\cite{AMMP1,AMMP2} that are closely relevant to this article were added in the references.

\end{document}